\documentclass[conference,onecolumn]{IEEEtran}
\IEEEoverridecommandlockouts

\usepackage{amsmath,amssymb,amsfonts}
\usepackage{algorithmic}
\usepackage{graphicx}
\usepackage{textcomp}
\usepackage{xcolor}
\usepackage{graphicx}
\usepackage{amsmath}
\usepackage{booktabs}
\usepackage{amsfonts}
\usepackage{hyperref}
\usepackage[english]{babel}
\usepackage{amsthm}
\usepackage{mathtools} 
\theoremstyle{definition}

\newtheorem{theorem}{Theorem}
\newtheorem{lemma}{Lemma}
\newtheorem{corollary}{Corollary}
\newtheorem{problem}{Problem}
\theoremstyle{remark}

\newcommand{\R}{\mathbb{R}}
\newcommand{\E}{\mathbb{E}}
\newcommand{\Prob}{\mbox{Prob}}

\newcommand{\bmtx}{\begin{bmatrix}}
\newcommand{\emtx}{\end{bmatrix}}
\newcommand{\bsmtx}{\left[\begin{smallmatrix}}
\newcommand{\esmtx}{\end{smallmatrix}\right]}

\title{Stochastic LQR Design With Disturbance Preview}

\author{\IEEEauthorblockN{Jietian Liu}
\IEEEauthorblockA{\textit{Electrical Eng. and Computer Science} \\
\textit{University of Michigan}\\
jietian@umich.edu}
\and
\IEEEauthorblockN{Laurent Lessard}
\IEEEauthorblockA{\textit{Mechanical and Industrial Eng.} \\
\textit{Northeastern University}\\
l.lessard@northeastern.edu}
\and
\IEEEauthorblockN{Peter Seiler}
\IEEEauthorblockA{\textit{Electrical Eng. and Computer Science} \\
\textit{University of Michigan}\\
pseiler@umich.edu}
}

\date{September 2024}

\begin{document}

\maketitle

\begin{abstract}
This paper considers the discrete-time, stochastic LQR problem with $p$ steps of disturbance preview information where $p$ is finite.  We first  derive the solution for this problem on a finite horizon with linear, time-varying dynamics and time-varying costs.  Next, we derive the solution on the infinite horizon with linear, time-invariant dynamics and time-invariant costs.  Our proofs rely on the well-known principle of optimality.  We provide an independent proof for the principle of optimality that relies only on nested information structure.  Finally, we show that the finite preview controller converges to the optimal noncausal controller as the preview horizon $p$ tends to infinity. We also provide a simple example to illustrate both the finite and infinite horizon results.

\end{abstract}



\section{Introduction}

This paper considers control design for a discrete-time, stochastic Linear Quadratic Regulator (LQR) problem with finite preview information of the disturbance. We consider
both the finite and infinite horizon problems.  Our specific formulation is motivated by recent work on regret-based control \cite{sabag21ACC,sabag21arXiv,sabag22arXiv,goel20arXiv,goel22CDC,goel22TAC}. These regret-based formulations often use a controller with full knowledge of the disturbance as a baseline for performance.  
Then they show that a causal, time-varying controller, updated with online optimization, can recover this baseline performance.  We show that the optimal stochastic LQR controller with finite preview can also recover this baseline performance as preview goes to infinity. Thus the performance of the noncausal controller can be recovered either via time-varying control with online optimization or using the stochastic LQR controller with an additional sensor to provide disturbance preview.

Recent work on regret-based control has also studied finite disturbance preview, e.g., in predictive or MPC-based settings \cite{Zhang2021,Lin2021}. These works typically 
consider bounded disturbances and derive non-asymptotic regret bounds. In contrast, we adopt a stochastic formulation with i.i.d.\ disturbances. Our preview controller explicitly depends on the previewed disturbance realizations while optimizing the expected cost associated with post-preview disturbances. Rather than deriving regret bounds, we establish asymptotic convergence of the finite-preview LQR controller to the corresponding noncausal optimal controller as the preview increases. 

There is a large literature on finite preview control for stochastic LQR/LQG and its closely related $H_2$ formulation. These results exist in both continuous-time
\cite{MoeljaMeinsma2006H2Control, MarroZattoni2005H2OptimalRejection, SentouhEtAl2011TheH2OptimalPreviewController, Tomizuka1975OptimalContinuousFinitePreviewProblem, Kojima1999LQPreviewSynthesis,LINDQUIST1968OnOptimalStochasticControlWithSmoothedInformation, Hac1992OptimalLinearPreviewControl, Peng1991optimal, YOSHIMURA1993AnActiveSuspensionModel, MARZBANRAD2004StochasticOptimalPreviewControl, HashikuraEtAl2020OnImplementationOfH2PreviewOutputFeedbackLaw} and discrete-time \cite{Hazell2010FrameworkForDiscreteTimeH2PreviewControl, Ting2021WindTunnelStudy, Zattoni2026H2OptimalDecouplingWithPreview, Tomizuka1975OptimalDiscreteFinitePreviewProblems, Tomizuka1976OptimalLinearPreviewControl, ROH1999StochasticOptimalPreviewControl, Louam1988OptimalControl, Yim2011DesignOfPreviewController, TomizukaRosenthal1979OnTheOptiamlDigitalStateVectorFeedbackController, TomizukaFung1980DesignOfDigitalFeedforwardPreviewControllers}. A common approach in discrete-time is to construct an augmented system with a chain of delays to store the 
disturbance preview information.  This technique simplifies the preview control design as the standard LQR/LQG/$H_2$ solutions can be applied to this augmented model \cite{SentouhEtAl2011TheH2OptimalPreviewController, Hazell2010FrameworkForDiscreteTimeH2PreviewControl, Ting2021WindTunnelStudy, Kojima1999LQPreviewSynthesis, LINDQUIST1968OnOptimalStochasticControlWithSmoothedInformation, Tomizuka1975OptimalDiscreteFinitePreviewProblems, Tomizuka1976OptimalLinearPreviewControl, ROH1999StochasticOptimalPreviewControl, Louam1988OptimalControl, Peng1991optimal, Yim2011DesignOfPreviewController, YOSHIMURA1993AnActiveSuspensionModel, MARZBANRAD2004StochasticOptimalPreviewControl, TomizukaRosenthal1979OnTheOptiamlDigitalStateVectorFeedbackController, TomizukaFung1980DesignOfDigitalFeedforwardPreviewControllers}.  The work in \cite{Hac1992OptimalLinearPreviewControl} is an exception as it use calculus of variations to directly derive necessary conditions for optimality.

LQR/LQG/$H_2$ with preview is a reasonable formulation in applications where sensors provide preview measurements of stochastic disturbances.  For example, wind turbines must maximize power capture and reduce structural loads in the presence of wind speed fluctuations. LIDARs mounted on the turbine nacelle can measure the incoming wind field thus improving power capture and load reduction \cite{ozdemir13,schlipf13,scholbrock16}.  In this setting, the wind speed fluctuations are stochastic but can be measured with preview. As a second example, active vehicle suspensions can be used to improve ride and handling qualities. Forward-looking radars measure the upcoming road profile.  These measurements can be used for the active suspension control.  Again, this setting can be modeled with a stochastic disturbance, i.e., the road profile, with preview measurements \cite{Tomizuka1976OptimalLinearPreviewControl,Hac1992OptimalLinearPreviewControl,ROH1999StochasticOptimalPreviewControl,Louam1988OptimalControl,YOSHIMURA1993AnActiveSuspensionModel,MARZBANRAD2004StochasticOptimalPreviewControl}.

Our paper makes the following contributions to this literature. First, we solve the stochastic LQR problems with finite disturbance preview on both finite and infinite time horizons (Sections\ref{sec:FHresult}
and \ref{sec:IHresult}). In principle, these solutions follow from existing stochastic LQR results by using the state augmentation procedure.  However, we provide an independent proof for the finite-horizon and infinite-horizon problems based on the principle of optimality.
Moreover, we provide a novel proof for the principle of optimality in this setting
(Section~\ref{sec:principleofoptim}). Our proof only relies on the assumption of nested information and may be of independent interest. Finally, we compare the performance of the finite disturbance preview controller against the optimal noncausal controller (Theorem 11.2.1 of
\cite{hassibi99}). The optimal noncausal controller has full knowledge of the future disturbance. We show that the finite preview controller converges, in a certain sense, to the optimal noncausal controller as the preview horizon tends to infinity  (Section~\ref{sec:NCresult}). Finally, we illustrate the results with a simple example (Section~\ref{sec:example}).

Two of the most related works are \cite{Hac1992OptimalLinearPreviewControl} and  \cite{ROH1999StochasticOptimalPreviewControl}. Calculus of variations is used in \cite{Hac1992OptimalLinearPreviewControl} to derive a deterministic preview control solution. The approach starts by deriving the optimal noncausal controller with full knowledge of future disturbance, similar to that given in Theorem 11.2.1 of \cite{hassibi99}. This optimal noncausal controller is then converted to a finite-preview controller by taking the expectation for unavailable future (post-preview) information which is zero.  In \cite{ROH1999StochasticOptimalPreviewControl}, an LQG problem is solved with a cost function that includes cross terms between the state and input. They address the finite preview problem using an augmented state approach. The results in
both \cite{Hac1992OptimalLinearPreviewControl} and  \cite{ROH1999StochasticOptimalPreviewControl}
are similar to those contained in our paper.  The key difference lies in the approach used to solve the problem. Specifically, our proof directly uses a version of the principle of optimality with nested information.  This eliminates the need for state augmentation and hence our optimal controller is directly expressed with dynamics of the same order of the plant (not including augmented states).  Moreover, our finite horizon derivation can directly incorporate non-zero initial conditions.

\section{Notation}
\label{sec:Notation}

Let $\R^n$ and $\R^{n\times m}$ denote the sets of real $n\times 1$ vectors and $n\times m$ matrices, respectively.  The superscript $\top$ denote  the matrix transpose.  Moreover, if $M\in \R^{n\times n}$ then $M\succ 0$ denotes that $M$ is a symmetric positive definite matrix.  Similarly, we use $\succeq$, $\prec$, and $\preceq$ for positive semidefinite, negative definite, and negative semidefinite matrices, respectively.
A matrix $A\in \R^{n\times n}$ is said to be Schur if the spectral radius is $<1$.  We also introduce random variables in the problem formulation and use $\E$ to denote the expectation. Finally, we denote the Kronecker delta $\delta_{ij}$ where $\delta_{ij}=1$ if $i=j$ and $\delta_{ij}=0$ otherwise.

\section{Problem Formulation}
\label{sec:problemform}

We formulate the stochastic Linear Quadratic Regulator (LQR) problem with preview information of the
disturbance.  The problem is stated on both finite and infinite horizons. 

Consider the  discrete-time, linear time-varying (LTV) plant $P$ with the following state-space representation:
\begin{align}
\label{eq:LTVPlant}
x_{t+1} = A_t \, x_t + B_{u,t}\, u_t+B_{w,t}\, w_t,
\end{align}
where $x_t \in \R^{n_x}$, $u_t \in \R^{n_u}$, and $w_t \in \R^{n_w}$ are the state, input, and disturbance at time $t$, respectively.  The state matrices have compatible dimensions, e.g. $A_t\in \R^{n_x\times n_x}$, and are defined on a finite horizon $t=0, 1,\ldots, T-1$ with $T<\infty$. We assume the disturbance is the independent and identically distributed (IID) noise: $\E[w_i]=0$ and $\E[w_iw_j^{\top}]=\delta_{ij} \, I$ for $i,j\in \{0,1,\ldots,T-1\}$.  We allow for the initial condition $x_0$ to be non-zero.


The goal is to design a controller to reject the effect of the disturbance.  We assume the controller has measurements of the state and $p$ steps of disturbance preview where $p\ge 0$. Specifically, the controller has access to the following information at time $t$:
\begin{align}
\label{eq:FHinfo}
\begin{split}    
    i_t & \coloneq  \{x_0,\dots,x_t,w_0, \dots,w_{t+p}\} \,\, \mbox{ if } t\le T-p-1 \\
    i_t & \coloneq  \{x_0,\dots,x_t,w_0, \dots,w_{T-1}\} \,\, \mbox{ otherwise.}
\end{split}
\end{align}
The available information depends on the amount of preview $p\ge 0$.  Also note that the information $i_t$ contains some redundant information.  However, it is important that this information pattern is nested, i.e. $i_0 \subseteq i_1 \subseteq \dots \subseteq i_{T-1}$. This allows the controller, in general, to access past information as time goes on. This nesting structure is also used later in our formal proofs.

We allow for controllers that can take random actions based on the available information.  Thus the controller policy at time $t$ is specified by a probability density function over possible actions given the available information:
\begin{align}
\label{eq:Kt}
  K_t(u_t,i_t) \coloneq  \Prob(u_t \, | \, i_t)
\end{align}
This general formulation includes deterministic policies as a special case, which corresponds
to functions of the form $u_t = k_t(i_t)$.
We'll denote the sequence of control policies by $K\coloneq \{K_0,\ldots,K_{T-1}\}$.
The performance is evaluated with a linear quadratic cost. In particular, the cost for a sequence of policies $K$ evaluated on a specific disturbance sequence $w$ and initial condition $x_0$ is defined as follows: 
\begin{align}
\label{eq:JT}
J_T(K,w,x_0)\coloneq x_T^\top Q_T x_T + \sum_{t=0}^{T-1} x_t^\top Q_t x_t + u_t^\top R_tu_t ,
\end{align}
where $Q_t \succeq 0$ for $t=0,\ldots, T$ and $R_t  \succ 0$ 
for $t=0,\ldots, T-1$. These matrices define the state and control costs.  

We now state the finite-horizon (FH), stochastic LQR
problem with preview information.
\begin{problem}[FH Stochastic LQR With Preview]
\label{prob:FHLQR}
Consider the LTV plant \eqref{eq:LTVPlant} defined on $t=0,\ldots,T-1$ and cost function in \eqref{eq:JT} with $Q_t \succeq 0$ for
$t=0,\ldots, T$ and
$R_t  \succ 0$ 
for $t=0,\ldots, T-1$.  Find a sequence of policies $K\coloneq \{K_0,\ldots,K_{T-1}\}$ 
to solve:
\begin{align}
     J_{T,p}^*(i_0)\coloneq \min_{K} \E\left[ J_T(K,w,x_0) \, | \, i_0 \right].
\end{align}
The policy $K_t$ at time $t$ uses information $i_t$ in \eqref{eq:FHinfo} and includes $p\ge 0$ steps of disturbance preview.
\end{problem}

We will also consider a similar problem on an infinite horizon.  In this case, we will assume the plant $P$
is linear, time-invariant (LTI):
\begin{align}
\label{eq:LTIPlant}
x_{t+1} = A \, x_t + B_{u}\, u_t+B_{w}\, w_t,
\end{align}
where $(A,B)$ are a stabilizable pair and $A$ is nonsingular\footnote{The solution to the infinite horizon problem (Section~\ref{sec:IHresult}) uses a related, discrete-time algebraic Riccati equation (DARE). There are technical issues in solving this DARE when $A$ is singular.  Our infinite horizon results can be extended to the case where $A$ is singular using generalized solutions for the DARE. See Section 21.3 and Remark 21.2 in \cite{Zhou1996Robust} for details.}. We will also assume a sequence of (possibly random) control policies
$K\coloneq \{K_0,K_1,\ldots\}$ with $p$ steps of disturbance preview:
\begin{align}
\label{eq:IHinfo}
    i_t & \coloneq  \{x_0,\dots,x_t,w_0, \dots,w_{t+p}\}. 
\end{align}
The cost function for a sequence of policies $K$ evaluated on a specific disturbance $w$ and initial condition $x_0$
is defined as the average per-step cost over the infinite horizon:
\begin{align}
\label{eq:Jinf}
J_\infty(K,w,x_0)\coloneq  \lim_{T\to \infty} \frac{1}{T} 
    \sum_{t=0}^{T-1} x_t^\top Q x_t + u_t^\top R u_t .
\end{align}
Here we assume $Q\succeq 0$, $R\succ 0$, and $(A,Q)$ is detectable.  The detectability assumption ensures that
any unstable modes in the plant will appear in the cost function.  We now state the infinite-horizon (IH), stochastic LQR problem with preview information. We restrict to policies that are stabilizing, i.e. the state remains uniformly bounded in both mean and second-moment.

\begin{problem}[IH Stochastic LQR With Preview]
\label{prob:IHLQR}
Consider the LTI plant \eqref{eq:LTIPlant} with $A$ nonsingular, $(A,B)$ stabilizable and cost function in \eqref{eq:Jinf} with $Q\succeq 0$, $R \succ 0$ and $(A,Q)$ detectable. Find a sequence of policies $K\coloneq \{K_0,K_1\dots\}$ 
to stabilize the plant and solve:
\begin{align}
    J_{\infty,p}^*(i_0)\coloneq \min_{K \mbox{stabilizing}} \E\left[ J_\infty(K,w,x_0) \, | \, i_0 \right].
\end{align} 
The policy $K_t$ at time $t$ uses information $i_t$ in \eqref{eq:IHinfo} and includes $p\ge 0$ steps of disturbance preview.
\end{problem}

The IH cost could depend on the initial information $i_0$ for some policies, e.g. a policy could lead to unbounded trajectories and infinite cost for some initial conditions $x_0$ but not others.  However, we will show in Section~\ref{sec:IHresult} that the optimal cost does not depend on the initial conditions, i.e. $J_{\infty,p}^*(i_0)=J_{\infty,p}^*$.

\section{Principle of Optimality}
\label{sec:principleofoptim}
Our solutions for the FH and IH problems with preview will rely on the principle of optimality.  This is a well known principle, e.g. see Section 2.2. of \cite{anderson2007optimal}.  Here we will state a specific version that can be applied to the problems with disturbance preview.  We state the principle of optimality for a more general, finite-horizon cost function:
\begin{align}
\label{eq:Jgeneral}
J(K,i_0)\coloneq \E\left[ g_T(x_T) + \sum_{t=0}^{T-1} g_t(x_t,u_t) \, \middle\vert \, i_0 \right],
\end{align}
where the per-step cost functions $g_0,\ldots,g_T$ are given.  In this section we drop the subscripts in the cost function $J$ to simplify the notation. The FH cost in the previous section is a special case where the $g_0,\ldots,g_T$ are quadratic functions. Moreover, we will consider a more general information structure for $i_t$. Specifically, we only assume that $i_t$ is nested, i.e. $i_t\subseteq i_{t+1}$. The information structures in the previous section have this nested structure. The goal is to solve $J^*(i_0) = \min_K J(K,i_0)$. 

We define the value function or optimal cost-to-go at time $t$ as follows:
\begin{align}
\label{eq:ValueFunction}
V_t(i_t)\coloneq \min_{K_t,\ldots,K_{T-1}}
\E\left[ g_T(x_T) + \sum_{j=t}^{T-1} g_j(x_j,u_j) \, \middle\vert \, i_t \right].
\end{align}
The value function at time $t=0$ corresponds to the optimal cost: $V_0(i_0)=J^*(i_0)$. The principle of optimality states that we can compute the value functions recursively, optimizing over one action at a time. 

\begin{theorem}[Principle of Optimality]
\label{thm:PrincipleOfOpt}
Define the value function in
\eqref{eq:ValueFunction} and assume the
information is nested:
$i_0 \subseteq i_1 \subseteq \dots \subseteq i_{T-1}$. Then the value function satisfies the following backwards recursion starting from $t=T$:
\begin{align}
V_T(i_T) &=  \E[g_T(\boldsymbol{x}_T)\, | \, i_T], \\
\label{eq:ValueIteration}
V_t(i_t)  &=  \min_{u} \E\left[g_t(x_t,u_t)+V_{t+1}(i_{t+1}) \, \middle| \, i_t, u_t=u \right] \\
\nonumber
&\quad \mbox{ for } t=0,\dots,T-1.
\end{align}
Moreover, the optimal cost $J^*(i_0)$ is achieved by a deterministic policy defined by selecting $u_t=k_t(i_t)$ at each $t$ to minimize the value function in \eqref{eq:ValueIteration}.
\end{theorem} 
\begin{proof}
Recall the value function defined in 
\eqref{eq:ValueFunction}. Now define a recursive version as
\begin{align*}
W_T(i_T) &\coloneq   \E[g_T(\boldsymbol{x}_T)\, | \, i_T], \\
W_t(i_t)  &\coloneq   \min_u \E\left[g_t(x_t,u_t)+W_{t+1}(i_{t+1}) \, \middle| \, i_t, u_t=u \right] \\
&\hphantom{:}\quad \mbox{ for } t=0,\dots,T-1.
\end{align*}
We will show, by induction, that $W_t=V_t$ for all $t$. The base case $W_T=V_T$ holds by definition. Next, make the induction assumption that $W_{t+1}=V_{t+1}$. We will prove that $W_t=V_t$ holds as well.  Consider the cost-to-go defined for any fixed set of policies $\{K_t,\dots,K_{T-1}\}$:
\begin{align}
V_t^{K_{t: T-1}}(i_t)  \coloneq 
\E\left[g_T(x_T)+\sum_{j=t}^{T-1} g_j(x_j,u_j)
\, \middle| \, i_t \right].
\end{align}
We can re-write this using the tower rule\footnote{We have $\E\bigl[\E[x|y]\bigr]=\E[x]$. More generally, whenever $y\subseteq z$, we have $\E\left[\, \E[x|z]\, \middle| \, y\, \right]=\E[x|y]$. See Theorem 9.1.5 in \cite{chung2001}. This is also known as the law of total expectation.}
\begin{align}
\label{eq:VkTower}
V_t^{K_{t: T-1}}(i_t)  
=\E\left[g_t(x_t,u_t) \, \middle| \, i_t \right]
+ \E\left[
\E\left[g_T(x_T)+\sum_{j=t+1}^{T-1} g_j(x_j,u_j)
\, \middle| \, i_{t+1} \right]
\, \middle| \, i_t \right].
\end{align}
The tower rule applies due to the assumption of nested information,  $i_t \subseteq i_{t+1}$. 
The second term depends on the policies
$\{K_{t+1},\dots,K_{T-1}\}$.  We can lower bound this term by the value function at time $t+1$. This gives:
\begin{align}
\label{eq:VkLowerBound1}
V_t^{K_{t: T-1}}(i_t)  
\ge \E\left[g_t(x_t,u_t)+
V_{t+1}(i_{t+1})  
\, \middle| \, i_t \right].
\end{align}
We have $W_{t+1}=V_{t+1}$ by the induction assumption. Next note that that the information $i_t$ is trivially a subset of the information
$\{i_t,u_t\}$. Hence another application of the tower rule gives:
\begin{align}
\label{eq:VkTower2}
V_t^{K_{t: T-1}}(i_t)  
\ge \E\left[ \vphantom{\Big|}
\E\left[ g_t(x_t,u_t)+
W_{t+1}(i_{t+1}) 
\, \middle| \, i_t, u_t=u \right]
\, \middle| \, i_t \right].
\end{align}
The right side depends on the policy $K_t$.  As stated previously, this policy is specified by a probability density function $K_t(u_t,i_t) \coloneq  \Prob(u_t \, | \, i_t)$. We can replace the outer expectation in \eqref{eq:VkTower2} by an integral over all control actions:
\begin{align}
V_t^{K_{t: T-1}}(i_t)  
\ge 
\int
\E\left[ g_t(x_t,u_t)+
W_{t+1}(i_{t+1}) 
\, \middle| \, i_t, u_t=u \right]
K_t(u,i_t) \, du.
\end{align}
Finally, the right side is lower bounded by  minimizing over the control input:
\begin{align}
\label{eq:VkLowerBound2}
V_t^{K_{t: T-1}}(i_t) 
\ge \min_u \E\left[g_t(x_t,u_t)+W_{t+1}(i_{t+1}) \, \middle| \, i_t, u_t=u \right].
\end{align}
The right side is, by definition, $W_t(i_t)$. Thus we have shown that $V_t^{K_{t: T-1}}(i_t)  \ge W_t(i_t)$
for any set of policies $\{K_t,\dots,K_{T-1}\}$.
This implies that the value function at time $t$ satisfies $V_t(i_t) \ge W_t(i_t)$.

In fact, we can find a set of policies 
$\{K_{t},\dots,K_{T-1}\}$ that
make $V_t(i_t) = W_t(i_t)$. For Equation~\ref{eq:VkLowerBound1}, we have equality if we pick $\{K_{t+1}, \ldots, K_{T-1} \}$ to be the optimal policies for the value function at $t+1$. For Equation~\ref{eq:VkLowerBound2}, we achieve equality by selecting a deterministic policy $u_t=k_t(i_t)$ with
\begin{align*}
  u_t & =\arg \min_u \E\left[g_t(x_t,u_t)+W_{t+1}(i_{t+1}) \, \middle| \, i_t, u_t=u \right]. 
\end{align*}
The policy that achieves $J^*(i_0)$ from will be deterministic
if we use this choice for $t=0$ to $T-1$.
\end{proof}

Note that the proof only relies on the assumption of nested information. It does not use any specific assumptions regarding the model / dynamics relating $x_{t+1}$ and $(x_t,u_t)$ (e.g., Markov assumption). It also does not rely on specific structure for the per-step cost functions $g_0,\ldots,g_T$. That said, we will apply this version of the principle of optimality in the next section for the specific case with linear dynamics, quadratic cost functions, and nested information with disturbance preview.

\section{Main Results}

\subsection{FH Stochastic LQR with Preview}
\label{sec:FHresult}

This section presents the solution to the FH stochastic
LQR problem with disturbance preview (Problem~\ref{prob:FHLQR}). We start with a standard technical lemma regarding a backwards Riccati iteration.

\begin{lemma}
\label{lem:RiccatiIter}
Let $\{A_t\}_{t=0}^{T-1} \subset \R^{n_x\times n_x}$,
$\{B_{u,t}\}_{t=0}^{T-1} \subset \R^{n_x\times n_u}$,
$\{Q_t\}_{t=0}^{T} \subset \R^{n_x\times n_x}$,
and
$\{R_t\}_{t=0}^{T-1} \subset \R^{n_u\times n_u}$
be given. Assume $Q_t \succeq 0$ and $R_t \succ 0$ for each $t$.
Define the following backwards Riccati iteration
for $t=0,\ldots,T$:
\begin{align}
\label{eq:RDE}
  &P_T\coloneq Q_T, \\
  \nonumber
  &P_t\coloneq Q_t+A_t^{\top}P_{t+1}A_t-A_t^{\top}P_{t+1}B_{u,t}H_t^{-1}B_{u,t}^{\top}P_{t+1}A_t,
\end{align}
where 
\begin{align}
\label{eq:Ht}
     H_t \coloneq R_t+B_{u,t}^{\top}P_{t+1}B_{u,t}
     \mbox{ for } t=0,\ldots, T-1
\end{align}
Then $P_t\succeq 0$ for $t=0,\ldots, T$.
and $H_t \succ 0$, for $t=0,\ldots,T-1$.
Hence $H_t$ is nonsingular for each $t$ and the Riccati iteration is well defined.
\end{lemma}
\begin{proof}
The proof is by induction. The base case is $P_T\succeq 0$ and $H_{T-1} \succ 0$.
This follows from the assumptions  $Q_T\succeq 0$ and $R_{T-1} \succ 0$.
Next, make the induction assumption that
$P_{t+1}\succeq 0$ and $H_t \succ 0$.
We will show that $P_t\succeq 0$ and $H_{t-1}\succ 0$. 
Use the generalized matrix inversion lemma (Fact 8.4.13 of \cite{bernstein2018scalar})
to rewrite the Riccati iteration: 

{\small
\begin{align*}
    P_t = Q_t 
+ A_t^\top P_{t+1} \left( P_{t+1} + P_{t+1} B_{u,t} R_t^{-1} B_{u,t}^\top P_{t+1}
\right)^{\dagger} P_{t+1} A_t,
\end{align*}
}

\noindent
where $^{\dagger}$ is the Moore-Penrose pseudo inverse. Thus $P_t\succeq 0$ follows from
$P_{t+1}\succeq 0$, $Q_t \succeq 0$ and $R_t \succ 0$. Moreover, $P_t \succeq 0$ combined with $R_{t-1} \succ 0$ imply $H_{t-1}\succ 0$.
\end{proof}

The next theorem provides the
solution to Problem~\ref{prob:FHLQR}. 
The proof is based on
the principle of optimality (Theorem~\ref{thm:PrincipleOfOpt}) and dynamic programming. 

\begin{theorem}
\label{thm:FHLQR}
Consider the FH stochastic LQR with preview including the assumptions stated in Problem~\ref{prob:FHLQR}. Define the
following feedback gains using the solution of backwards Riccati iteration \eqref{eq:RDE}:
\begin{align}
\label{eq:KxFH}
   K_{x,t} &\coloneq H_t^{-1}B_{u,t}^{\top}P_{t+1}A_t, \\
\label{eq:KwFH}
   K_{w,t} &\coloneq H_t^{-1}B_{u,t}^{\top}P_{t+1}B_{w,t},\\
\label{eq:KvFH}
   K_{v,t} &\coloneq H_t^{-1}B_{u,t}^{\top}. 
\end{align}
The sequence of policies that achieve $J_{T,p}^*(i_0)$ are deterministic, and have the form: 
\begin{align}
\label{eq:FHustar}
  u_t^*=-K_{x,t}x_t-K_{w,t}w_t-K_{v,t}v_{t+1},  
\end{align}
where $v_{t+1}$ depends on $\{w_{t+1},\ldots,w_{t+p} \}$ and is given by:
\begin{align}
\label{eq:vFH}
\begin{split}
& v_{t+1} = \sum_{j=t+1}^{t+p}
\left[ \hat{A}_{t+1}^\top
\cdots
\hat{A}_j^\top
\right]
P_{j+1} B_{w,j} w_j \\
& \mbox{ with }
\hat{A}_j\coloneq A_j-B_{u,j}K_{x,j}.
\end{split}
\end{align}
Equation~\ref{eq:vFH} uses the convention that
$w_j=0$ when $j\geq T$.
\end{theorem}

\begin{proof}  
The assumptions in 
Problem~\ref{prob:FHLQR} are sufficient for Lemma~\ref{lem:RiccatiIter} to hold. Hence the Riccati iteration and the feedback  gains \eqref{eq:KxFH}-\eqref{eq:KvFH} are well-defined. Moreover $H_t\succ 0$ for each $t$.

Define the value function or optimal cost-to-go at time $t$ for Problem~\ref{prob:FHLQR}  as follows:

{\small
\begin{align}
\label{eq:ValueFunctionLQR}
V_t(i_t)\coloneq \min_{K_t,\ldots,K_{T-1}}
\E\left[ x_T^\top Q_T x_T + \sum_{j=t}^{T-1} x_j^\top Q_j x_j + u_j^\top R_j u_j \, \middle\vert \, i_t \right].
\end{align}
}

\noindent
The value function at time $t=0$ corresponds to the optimal cost: $V_0(i_0)=J_{T,p}^*(i_0)$. 
The value functions can be computed recursively by the principle of optimality (Theorem~\ref{thm:PrincipleOfOpt}):

{\small
\begin{align}
\nonumber
   V_T(i_T)= & \E[x_T^{\top}Q_Tx_T \, | \, i_T] \\
\nonumber
   V_t(i_t) = &  \min_u \E\left[ x_t^{\top}Q_tx_t + u_t^{\top}R_tu_t + V_{t+1}(i_{t+1})
        \, \middle| \, i_t, u_t=u \right] \\
    \label{eq:ValueIterationLQR}
        &   \mbox{ for } t=0,\dots,T-1.
\end{align}
}

\noindent
We will show that the value functions have the form: 
\begin{align}
\label{eq:VFHquadratic}
 V_t(i_t)=x_t^{\top}P_t x_t + 2\bar{v}_t^{\top}x_t+q_t
 \mbox{ for } t=0,\dots,T-1,
\end{align}
where $\{P_t\}_{k=0}^T$ satisfy the backwards Riccati recursion
\eqref{eq:RDE}, $q_t$ depend on $\{w_t,\ldots, w_{t+p}\}$, and $\bar{v}_t$ is given by:
\begin{align}
\label{eq:vbarFH}
\bar{v}_t \coloneq  \sum_{j=t}^{t+p}
\left[ \hat{A}_{t}^\top
\cdots
\hat{A}_j^\top
\right]
P_{j+1} B_{w,j} w_j.
\end{align}
Equation~\ref{eq:vbarFH} again uses the convention that
$w_j=0$ if $j\geq T$.~\footnote{Note that $\bar{v}_t$ depends on $\{w_t,\ldots,w_{t+p}\}$. This means that
$\bar{v}_{t+1}$ depends on $\{w_{t+1},\ldots,w_{t+p+1}\}$. However, $v_{t+1}$ defined in
\eqref{eq:vFH} depends on
$\{w_{t+1},\ldots,w_{t+p}\}$. 
Hence $\bar{v}_{t+1}$ and $v_{t+1}$ differ by one term.} We will show that \eqref{eq:VFHquadratic} holds by induction. The base case at $t=T$ corresponds to 
$P_T=Q_T$, $\bar{v}_T=0$, and $q_T=0$. Next make the induction assumption that the value function at time $t+1$ has the given form:
\begin{align}
\label{eq:ValueLQRtp1}
 V_{t+1}(i_{t+1})=x_{t+1}^{\top}P_{t+1} x_{t+1} + 2\bar{v}_{t+1}^{\top}x_{t+1}+q_{t+1}.
\end{align}
We will prove that the value function at time $t$ has the similar form given in \eqref{eq:VFHquadratic}.
Substitute~\eqref{eq:ValueLQRtp1} into the value function recursion
\eqref{eq:ValueIterationLQR} and use the LTV dynamics \eqref{eq:LTVPlant}
to replace $x_{t+1}$. This yields:

{\small
\begin{align}
\begin{split}
    \label{eq:FHVminu}
V_t(i_t)  = \min_{u} \E[x_t^{\top}Q_tx_t+u_t^{\top}R_tu_t
 +(A_tx_t+B_{u,t}u_t+B_{w,t}w_t)^{\top}P_{k+1}(A_tx_t+B_{u,t}u_t+B_{w,t}w_t)\\
+2\bar{v}_{t+1}^{\top}(A_tx_t+B_{u,t}u_t+B_{w,t}w_t)+q_{t+1} \, | \, i_t, \, u_t=u].
\end{split}
\end{align}
}

\noindent
Note that the quadratic term in this minimization is $u_t^\top H_t u_t$ where $H_t\succ 0$. Thus the  minimization has a unique  minima because the objective is a strictly convex function. Take the gradient respect to $u$ to find the optimal input:
\begin{align}
\label{eq:FHustar2}
    u_t^* = \E \left[  -K_{x,t} x_t - K_{w,t} w_t - K_{v,t} \bar{v}_{t+1}
\, \middle| \, i_t    
    \right].
\end{align}
where the gains are defined in \eqref{eq:KxFH}-\eqref{eq:KvFH}.
Note that $(x_t,w_t)$ are contained in $i_t$ so $E[x_t\, | i_t]=x_t$
and $E[w_t\, | i_t]=w_t$.
Moreover, the induction assumption gives:
\begin{align}
    \bar{v}_{t+1}= \sum_{j=t+1}^{t+p+1}
\left[ \hat{A}_{t+1}^\top
\cdots
\hat{A}_j^\top
\right]
P_{j+1} B_{w,j} w_j.
\end{align}
Thus $E[\bar{v}_{t+1}\, |  i_t] = v_{t+1}$ and the optimal input in \eqref{eq:FHustar2} simplifies to the expression in \eqref{eq:FHustar}.

Next, substitute $u^*_t$ back in to the cost-to-go function~\eqref{eq:FHVminu}.
After some simplification, the value function at time $t$ has the form given in \eqref{eq:VFHquadratic} with
$P_t$ given by the backwards Riccati recursion \eqref{eq:RDE} and $\bar{v}_t$ given by~\eqref{eq:vbarFH}. Moreover,
$q_t$ depends on $\{w_t,\ldots, w_{t+p}\}$ and can be simplified to:
\begin{align*}
q_t = & -(K_{w,t} w_t + K_{v,t} v_{t+1})^\top H_t   
    (K_{w,t} w_t + K_{v,t} v_{t+1}) 
+ w_t^\top B_{w,t}^\top \, (P_{t+1} B_{w,t} w_t + 2v_{t+1})
+ \E[q_{t+1}\, | \, i_t ].
\end{align*}
Hence the proof is complete by induction.
\end{proof}

The optimal cost, based on the proof, is
\begin{align}
J_{T,p}^*(i_0) = x_0^\top P_0 x_0 
+ 2 \bar{v}_0^\top x_0
+ q_0.    
\end{align}
The first term depends only on the state initial condition $x_0$.  The second term depends on both the $x_0$ and the initial disturbance information $\{w_0,...,w_{p}\}$. The third term also depends on the initial disturbance information but also
includes the expected cost of the disturbances $\{w_{p+1},...,w_{T-1}\}$.

\subsection{IH Stochastic LQR with Preview}
\label{sec:IHresult}

This section presents the solution to the IH stochastic LQR problem with disturbance preview (Problem~\ref{prob:IHLQR}). We start with a standard technical lemma regarding a discrete time algebraic Riccati equation (DARE).

\begin{lemma}
\label{lem:DARE}
Let $A \in \R^{n_x\times n_x}$,
$B_{u} \in \R^{n_x\times n_u}$,
$Q \in \R^{n_x\times n_x}$,
and
$R \in \R^{n_u\times n_u}$
be given.  Assume: 
(i) $Q\succeq 0$ and $R \succ 0$, (ii) $\left(A, B_u\right)$ stabilizable, (iii) $A$ is nonsingular, and (iv) $(A, Q)$ has no unobservable modes on the unit circle. Then there is a unique stabilizing solution $P \succeq 0$ such that: 
\begin{enumerate}
\item $P$ satisfies the following DARE:
\begin{equation}
\label{eq:DARE}
0 =P-A^{\top} P A-Q +A^{\top} P B_u H^{-1} B_u^{\top} P A,
\end{equation}
where $H\coloneq R+B_u^{\top} P B_u \succ 0$.

\item The gain $K_x\coloneq H^{-1}B_u^{\top} P A$ is stabilizing, i.e. $A-B_u K_x$ is a Schur matrix.

\end{enumerate}

\end{lemma} 

\begin{proof} Statements 1) and 2) follow from Corollary 21.13 and Theorem 21.7 of \cite{Zhou1996Robust} (after aligning the notation). 

\end{proof}

The next theorem constructs the IH stochastic LQR with preview controller using the stabilizing solution of the DARE.

\begin{theorem}
\label{thm:IHLQR}
Consider the IH stochastic LQR with preview including the assumptions stated in Problem~\ref{prob:IHLQR}. Define the
following feedback gains using the solution of DARE \eqref{eq:DARE}:
\begin{align}
\label{eq:KxIH}
   K_x &\coloneq H^{-1}B_u^{\top}P A, \\
\label{eq:KwIH}
   K_w &\coloneq H^{-1}B_u^{\top}PB_w,\\
\label{eq:KvIH}
   K_v &\coloneq H^{-1}B_u^{\top}. 
\end{align}
The sequence of policies that achieve $J_{\infty,p}^*(i_0)$ are deterministic, and have the form: 
\begin{align}
\label{eq:IHustar}
  u_t^*=-K_x x_t-K_w w_t-K_v v_{t+1},  
\end{align}
where $v_{t+1}$ depends on $\{w_{t+1},\ldots,w_{t+p} \}$ and is given by:
\begin{align}
\label{eq:vIH}
\begin{split}
& v_{t+1} = \sum_{j=t+1}^{t+p}
\left[ \hat{A}^\top \right]^{j-t}
P B_w w_j \\
& \mbox{ with }
\hat{A}\coloneq A-B_u K_x.
\end{split}
\end{align}

\end{theorem}
\begin{proof}  
Recall that the infinite horizon cost for any sequence of policies $K\coloneq \{K_0,K_1,\ldots\}$ is given by
\begin{align}
J_\infty(K,w,x_0)\coloneq  \lim_{T\to \infty} \frac{1}{T} 
    \sum_{t=0}^{T-1} g(x_t,u_t) ,
\end{align}
where $g(x_t,u_t)\coloneq x_t^\top Q x_t + u_t^\top R u_t$ is the per-step cost. The main part of the proof is to re-arrange the per-step cost into a form that involves $u_t-u_t^*$. This will be used to show that $u_t=u_t^*$ is the optimal input.

First, substitute for $Q$ using the DARE
\eqref{eq:DARE} and re-arrange to express the per-step cost as follows:
\begin{align*}
g(x_t, u_t)
= & (u_t+K_xx_t)^{\top}H(u_t+K_xx_t)
+x_t^{\top}Px_t 
-(Ax_t+B_uu_t)^{\top}P(Ax_t+B_uu_t).
\end{align*}
Use the LTI dynamics
\eqref{eq:LTIPlant} to replace
$Ax_t + B_u u_t$ by $x_{t+1}-B_w w_t$.
This yields
\begin{align*}        
g(x_t, u_t) =(u_t+K_xx_t)^{\top}H(u_t+K_xx_t) 
 +(x_t^{\top}Px_t- 
  x_{t+1}^{\top}Px_{t+1}) -w_t^{\top}B_w^{\top}PB_ww_t+2x_{t+1}^{\top}PB_w w_t.
\end{align*}
Next, use the definition
of $u_t^*$ in \eqref{eq:IHustar}
to replace $K_xx_t$
with $-u_t^*-K_ww_t-K_v v_{t+1}$.
The per-step cost is thus given by

{\small
\begin{align}
\label{eq:perstepwithlt}
g(x_t, u_t) =(u_t-u_t^*)^{\top}H(u_t-u_t^*)+(x_t^{\top}Px_t-x_{t+1}^{\top}Px_{t+1}) + l_t 
-w_t^{\top}B_w^{\top}PB_ww_t -(K_ww_t+K_vv_{t+1})^{\top}H(K_ww_t+K_vv_{t+1}),
\end{align}
}

\noindent
where
\begin{align*}
l_t \coloneq 
-2(K_ww_t+K_vv_{t+1})^{\top}H(u_t+K_xx_t)+2x_{t+1}^{\top}PB_ww_t.   
\end{align*}

We now focus on simplifying the term $l_t$. We have, from
\eqref{eq:KwIH} and \eqref{eq:KvIH}, that $HK_w=B_u^\top P B_w$ and $HK_v = B_u^\top$. Hence $l_t$ simplifies to:
\begin{align*}
l_t \coloneq 
-2(PB_ww_t+v_{t+1})^{\top}B_u(u_t+K_xx_t)
+ 2x_{t+1}^{\top}PB_ww_t.   
\end{align*}
We can replace $B_u(u_t+K_xx_t)$ with $x_{t+1}-\hat{A}x_t-B_ww_t$. Combining terms yields:
\begin{align*}
l_t \coloneq 
2(PB_ww_t+v_{t+1})^{\top}
(\hat{A}x_t+B_ww_t)
 -2 x_{t+1}^\top v_{t+1}.
\end{align*}
Note that the preview term $v_{t+1}$
defined in \eqref{eq:vIH} satisfies the following relation:
\begin{align} 
\hat{A}^\top \left( PB_w w_t + v_{t+1} \right)
= v_t + \left[ \hat{A}^\top \right]^{p+1} P B_w w_{t+p}.
\end{align}
We can use this to finally simplify $l_t$ to:
\begin{align*}
l_t =
2x_t^{\top}v_{t} 
+2x_t^{\top}\left[\hat{A}^\top \right]^{p+1}PB_ww_{t+p}
+ 2w_t^{\top}B_w^{\top}PB_ww_t
+2v_{t+1}^{\top}B_ww_t
-2x_{t+1}^{\top}v_{t+1}.
\end{align*}

Substitute this expression for $l_t$ back into the per-step cost \eqref{eq:perstepwithlt}. Re-organizing terms gives:
\begin{subequations}
\begin{align}
& g(x_t,u_t) =(u_t-u_t^*)^{\top}H(u_t-u_t^*) \\
&-(K_ww_t+K_vv_{t+1})^{\top}H(K_ww_t+K_vv_{t+1}) \\
\label{eq:TelescopingTerms}
&+(x_t^{\top}Px_t-x_{t+1}^{\top}Px_{t+1}) 
+2(x_t^{\top}v_{t}-x_{t+1}^{\top}v_{t+1}) \\
\label{eq:perstepTerm1}
&+2x_t^{\top}\left[ \hat{A}^\top \right]^{p+1}PB_ww_{t+p} \\           
\label{eq:perstepTerm2}
&+2v_{t+1}^{\top}B_ww_t\\
&+w_t^\top B_w^\top PB_ww_t.
\end{align}
\end{subequations}
The terms \eqref{eq:TelescopingTerms} form telescoping sums when inserted into the cost  $J_\infty(K,w,x_0)$.
The telescoping sums contribute zero to the cost as $T\to \infty$ since the policy is assumed to be stabilizing. The expectation of term \eqref{eq:perstepTerm2} is zero as $v_{t+1}$ depends on $\{w_{t+1},\dots,w_{t+p}\}$ and is independent of $w_t$. The expectation of term \eqref{eq:perstepTerm1} is also zero as $x_t$ depends on information $i_{t-1}=\{x_0,\dots,x_{t-1},w_0,\dots,w_{t+p-1}\}$, which is independent of $w_{t+p}$. Thus the cost  for any stabilizing policy simplifies to:

{\small
\begin{align}
\label{eq:JinfSimplified}   
J_{\infty,p}^*(i_0)= \min_{K \mbox{stabilizing}} \E\biggl[\lim_{T\to \infty} \frac{1}{T}\sum_{t=0}^{T-1} (u_t-u_t^*)^{\top}H(u_t-u_t^*)
-(K_ww_t+K_vv_{t+1})^{\top}H(K_ww_t+K_vv_{t+1})
 +w_t^\top B_w^\top PB_ww_t
\, \Big\vert \, i_0
\biggr].
\end{align}
}

\noindent
Only the first term depends on the  plant input and $H\succ 0$. Thus $J_{\infty,p}^*(i_0)$ is achieved by $u_t=u_t^*$. 
\end{proof}

The optimal IH cost is obtained by substituting $u_t=u_t^*$ into
\eqref{eq:JinfSimplified}. This yields:
\begin{align*}    
    J_{\infty,p}^*(i_0)=\lim_{T\to \infty} \frac{1}{T} \E\biggl[\sum_{t=0}^{T-1} 
    w_t^\top B_w^\top PB_ww_t 
 -(K_ww_t+K_vv_{t+1})^{\top}H(K_ww_t+K_vv_{t+1})
\, \Big\vert \, i_0
\biggr].
\end{align*}
This can be simplified further using the definition of $v_{t+1}$ in \eqref{eq:vIH} combined with
$\E[w_iw_j^{\top}]=\delta_{ij} \, I$.  This gives: 
\begin{align}    
\label{eq:JinfFinal}   
    J_{\infty,p}^*=   \mbox{trace} \left[ 
    P B_w B_w^\top  \right]  
    - \mbox{trace} \left[ \sum_{j=0}^p H K_v (\hat{A}^\top)^j P B_w B_w^\top P \hat{A}^j K_v^\top \right].     
\end{align}
Here we have used $K_w = K_vPB_w$ to simplify the expression further. Note that the optimal IH cost does not depend on the initial condition, i.e. $J_{\infty,p}^*(i_0) = J_{\infty,p}^*$.  The first term in the optimal cost, trace$[PB_wB_w^\top]$, is the cost that would be achieved using only state feedback. The second term gives the cost reduction obtained by using the current disturbance measurement and $p$ steps of preview.

\subsection{Relation to Optimal Noncausal Controller}
\label{sec:NCresult}

This section discusses the optimal noncausal controller with full knowledge of the (past, current and future) values of the disturbance.  
We focus on the IH case where the
controller has access to the following
information at time $t$:
\begin{align}
\label{eq:NCinfo}
    i_t & \coloneq  \{x_0,\dots,x_t,w_0, w_1,\ldots\}. 
\end{align}
The disturbance is fully known so we consider deterministic policies
$K\coloneq \{k_0,k_1,\ldots\}$ where $u_t=k_t(i_t)$. Moreover, we assume $w\in \ell_2$ and the (deterministic) cost is:
\begin{align}
\label{eq:JinfNC}
\sum_{t=0}^\infty x_t^\top Q x_t + u_t^\top R u_t.
\end{align}
Here the cost is deterministic (in contrast to the expected costs defined previously).
The next problem states the IH design with full preview.

\begin{problem}[IH LQR With Full Preview]
\label{prob:NCLQR}
Consider the LTI plant \eqref{eq:LTIPlant} with $A$ nonsingular, $(A,B)$ stabilizable and cost function in \eqref{eq:JinfNC} with $Q\succeq 0$, $R \succ 0$ and $(A,Q)$ detectable. Also assume $w\in \ell_2$. Find a sequence of deterministic policies $K\coloneq \{k_0,k_1\dots\}$ 
to stabilize the plant and solve:
\begin{align}
    \min_{K \mbox{stabilizing}} \sum_{t=0}^\infty x_t^\top Q x_t + u_t^\top R u_t.
\end{align} 
The policy $k_t$ at time $t$ uses information $i_t$ in \eqref{eq:NCinfo} and includes full preview of the disturbance.
\end{problem}

A solution for the optimal noncausal controller is given in Theorem 11.2.1 of
\cite{hassibi99}. Related  noncausal results (both FH and IH) are given in \cite{sabag21ACC,sabag21arXiv,sabag22arXiv,goel20arXiv,goel22CDC} where the noncausal controller is used as a baseline for regret-based control design.  The next result gives the IH noncausal controller in a form/notation that closely aligns with the finite preview controller in Theorem~\ref{thm:IHLQR}. 

\begin{theorem}
\label{thm:IHNoncausal}
Consider IH LQR with full preview 
including the assumptions stated in Problem~\ref{prob:NCLQR}. Define the
feedback gains $(K_x,K_w,K_v)$
in \eqref{eq:KxIH}-\eqref{eq:KvIH}
using the solution of DARE \eqref{eq:DARE}.
The sequence of policies that achieve the minimum are deterministic and have the form: 
\begin{align}
\label{eq:NCustar}
  u_t^{nc}=-K_x x_t-K_w w_t-K_v v^{nc}_{t+1},  
\end{align}
where $v^{nc}_{t+1}$ depends on $\{w_{t+1},w_{t+2},\ldots \}$ and is given by the following anticausal dynamics:
\begin{align}
\label{eq:vNC}
\begin{split}
& v^{nc}_t = \hat{A}^\top (v^{nc}_{t+1} + P B_w w_t),
\,\, v^{nc}_\infty = 0 \\
& \mbox{ with }
\hat{A}\coloneq A-B_u K_x.
\end{split}
\end{align}
\end{theorem}

This result, as stated, is a special case of Theorem 1 in \cite{liu2024robust}.  It was stated in \cite{liu2024robust} for a more general LQR cost including a cross term $2x_t^\top S w_t$.\footnote{The proof given in \cite{liu2024robust} was for signals defined for $t$ from $-\infty$ to $+\infty$. However, the optimal noncausal controller is the same for signals defined from $t=0$ to $\infty$.}  Note that the full preview design (Problem~\ref{prob:NCLQR}) was stated with a deterministic formulation while finite preview design (Problem~\ref{prob:IHLQR})  was derived with stochastic formulation. 
However, we show next that the controller with finite preview $p<\infty$ converges, in the limit as $p\to \infty$, to the optimal noncausal controller for any fixed disturbance $w\in \ell_2$.

\begin{theorem}
\label{thm:ControllerConv}
Let $w\in \ell_2$ be any given (deterministic) disturbance sequence.  Then the optimal noncausal controller is given by
$u_t^{nc}$ in \eqref{eq:NCustar} where $v^{nc}_{t+1}$ can be expressed as:
\begin{align}
\label{eq:vNC2}
v^{nc}_{t+1} = \sum_{j=t+1}^\infty
\left[ \hat{A}^\top \right]^{j-t}
P B_w w_j
\end{align}
Moreover, let $u_t^p$ denote the optimal finite preview controller given in \eqref{eq:IHustar} with $p<\infty$ steps of preview.  Then
$\|u_t^{nc}-u_t^p\|_2 \to 0$ uniformly in $t$ as $p\to \infty$.
\end{theorem}
\begin{proof}
We will first verify that \eqref{eq:vNC2} holds. Specifically, we will show that $v_{t+1}^{nc}$ defined by \eqref{eq:vNC} is equal to the following signal:
\begin{align}
   y_{t+1} = \sum_{j=t+1}^\infty
    \left[ \hat{A}^\top \right]^{j-t}
    P B_w w_j
\end{align}
It can be verified, by direct substitution, that $y_{t+1}$ satisfies the state-space recursion in \eqref{eq:vNC}. Moreover, 
$y_{t+1}$ satisfies the boundary condition $y_\infty=0$ because $w\in \ell_2$ and $\hat{A}$ is a Schur matrix.\footnote{We can bound $y_{t+1}$  by $\|y_{t+1}\|_2 \le a \cdot b_t$ where $a:=\left\| \sum_{l=1}^\infty
[ \hat{A}^\top]^l P B_w \right\|_{2\to 2}$ and $b_t:=
\max_{j\ge t+1} \| w_j\|_2$.
We have $a<\infty$ because $\hat{A}$ is a Schur matrix. 
Moreover, $b_t\to 0$ as $t\to\infty$ because $w\in \ell_2$. Hence $y_{t+1}\to 0$ as $t\to \infty$.}
Thus $y_{t+1}$ is the unique solution to the noncausal dynamics in \eqref{eq:vNC} and hence \eqref{eq:vNC2} holds.

Next, we show that $\|u_t^{nc}-u_t^p\|_2 \to 0$ uniformly in $t$ as $p\to \infty$.  Let $x_t^{nc}$ and $x_t^p$ denote the response of the LTI system
\eqref{eq:LTIPlant} with some initial condition $x_0$ and disturbance $w\in\ell_2$ using the noncausal and $p$-step finite preview controller, respectively.
The difference between the two commands is:
\begin{align}
\label{eq:diffu}
   u_t^{nc} - u_t^p = -K_x(x_t^{nc}-x_t^p)
    -K_v (v_{t+1}^{nc}-v_{t+1}^p).
\end{align}
Next, note that the states satisfy the following recursion:
\begin{align*}
x_{t+1}^{nc} - x_{t+1}^p & = 
A \, (x_t^{nc} - x_t^p ) +  B_u \, (u_t^{nc} - u_t^p ) \\
& = \hat{A} (x_t^{nc} - x_t^p ) -  B_u K_v \, (v_{t+1}^{nc} - v_{t+1}^p ),
\end{align*}
where we have used \eqref{eq:diffu} to substitute for $u_t^{nc}-u_t^p$. Iterating from $x_0^{nc}-x_0^p = 0$ (since the initial conditions match) gives:
\begin{align*}
x_t^{nc} - x_t^p & = -\sum_{j=1}^t \hat{A}^{t-j} B_u K_v (v_j^{nc} - v_j^p ).
\end{align*}
Substitute this back into \eqref{eq:diffu} to express the difference in the control commands as follows:
\begin{align}
\label{eq:diffu2}
\begin{split}    
   u_t^{nc} - u_t^p =
     -K_v (v_{t+1}^{nc}-v_{t+1}^p) 
     +K_x \sum_{j=1}^t \hat{A}^{t-j} B_u K_v (v_j^{nc} - v_j^p ).
\end{split}
\end{align}
To bound the right side, we use \eqref{eq:vNC2} to
rewrite $(v_{t+1}^{nc}-v_{t+1}^p)$ in terms of the disturbance:
\begin{align}
    v_{t+1}^{nc}-v_{t+1}^p= \sum_{l=p+1}^\infty
( \hat{A}^\top)^{l} P B_w w_{l+t}.
\end{align}
Thus,
$\| v_{t+1}^{nc}-v_{t+1}^p \|_2  
\le a_p \cdot b$ for $t\geq0$, where
\begin{align*}
 a_p  :=\left\| \sum_{l=p+1}^\infty
[ \hat{A}^\top]^{l} P B_w \right\|_{2\to 2} \mbox{ and }
b  := \max_{j} \| w_j\|_2.   
\end{align*}
The assumption $w\in \ell_2$ implies that $b$ is finite.  Moreover, $\hat{A}$ is a Schur matrix and hence $a_p\to 0$ as $p\to \infty$.\footnote{The sum 
$\sum_{l=p+1}^\infty
[ \hat{A}^\top]^{l}$ converges to $(I-\hat{A}^\top)^{-1} \, (\hat{A}^\top)^{p+1}$. 
Thus,
\begin{align*}
a_p \le \| (I-\hat{A}^\top)^{-1} \|_{2\to 2}
\cdot
\|(\hat{A}^\top)^{p+1} \|_{2\to 2}
\cdot
\| P B_w \|_{2 \to 2}.
\end{align*}
We have $(\hat{A}^\top)^{p+1}\to 0$ as $p\to \infty$ and hence $a_p\to 0$. }

Finally, use this bound and Equation~\ref{eq:diffu2} to bound the difference in control commands:
\begin{align}
    \|u_t^{nc} - u_t^p\|_2 \le a_p b \left[ 
    c_1 + c_2
    \cdot \sum_{j=0}^{t-1}\| \hat{A}^j \|_{2\to 2}
    \right],
\end{align}
where $c_1:=\|K_v\|_{2\to 2}$ and 
$c_2:=\|K_x \|_{2\to 2} \cdot \|B_u K_v\|_{2\to 2}$
are finite constants. 
We can choose any $\alpha\in(\rho(\hat{A}),\,1)$ and $\exists \, T$ such that $\left\|\hat{A}^j\right\|\leq\alpha^j$
for all $j\ge T$. This follows from
Gelfand's formula (Corollary 5.6.14 in \cite{horn2013matrix}). Therefore, the sum $\sum_{j=0}^\infty\| \hat{A}^j \|_{2\to 2}$ converges to some finite number $c_3$.   Thus we can bound the difference in control commands as:
\begin{align}
\|u_t^{nc} - u_t^p\|_2 \le a_p b \, (c_1+c_2 \, c_3),  
\end{align}
where $b$, $c_1$, $c_2$, and $c_3$ are all finite constants. Moreover $a_p \to 0$ as $p\to \infty$.
This implies that $\| u_t^{nc} - u_t^p \|_2 \to 0$ uniformly in $t$ as $p\to \infty$.

    
\end{proof}

Finally, we can derive a related convergence result for the costs achieved with finite preview and full, noncausal preview. 
The optimal noncausal controller depends on the specific (deterministic) disturbance sequence and assumes $w\in \ell_2$ to ensure the cost \eqref{eq:JinfNC} is bounded. However, we can apply the noncausal controller to stochastic disturbances $w$ that are drawn from the distribution of IID, zero mean, unit variance sequences.  The expected per-step cost for the optimal noncausal controller is:
\begin{align}
\label{eq:JinfExpNC}
    J_{\infty,nc}^*(x_0)\coloneq \E\left[ J_\infty(K_{nc},w,x_0) \, | \, x_0 \right],
\end{align}
where $K_{nc}$ is the optimal noncausal
controller \eqref{eq:NCustar} for a specific  sequence $w$.  The expected per-step cost for the noncausal controller is averaged over stochastic signals $w$. We show below that it does not depend on the initial state, i.e. $J_{\infty,nc}^*(x_0)=J_{\infty,nc}^*$. The next result is that the expected per-step cost with the optimal finite preview controller converges to the cost achieved by the optimal noncausal controller. We also give a simple expression to compute the optimal cost achieved by the noncausal controller.


\begin{theorem}
\label{thm:CostConv}
Let $w$ be an IID sequence with $\E[w_i]=0$ and $\E[w_iw_j^{\top}]=\delta_{ij} \, I$ for $i,j\in \{0,1,\ldots\}$.  The expected per-step cost for the optimal noncausal controller is:
\begin{align}    
\label{eq:JinfNCSimple}
    J_{\infty,nc}^*=  & \mbox{trace} \left[ 
    P B_w B_w^\top  \right] 
    - \mbox{trace} \left[ 
    H K_v X K_v^\top \right]. 
\end{align}
where $X$ is the solution to the following discrete-time Lyapunov equation:
\begin{align}
\label{eq:NCLyap}
    \hat{A}^\top X \hat{A} - X
    + P B_w B_w^\top P = 0
\end{align}
Moreover, let $J_{\infty,p}^*$ denote the optimal expected per-step cost in
\eqref{eq:JinfFinal} achieved by the optimal IH controller with $p<\infty$ steps of preview.  Then $J_{\infty,p}^* \to J_{\infty, nc}^*$ as $p\to \infty$.
\end{theorem}
\begin{proof}
The steps in the proof of Theorem~\ref{thm:IHLQR} can be used to show that the noncausal cost is given by:
\begin{align}    
    J_{\infty,nc}^*=  & \mbox{trace} \left[ 
    P B_w B_w^\top  \right]  
     - \mbox{trace} \left[ \sum_{j=0}^\infty H K_v (\hat{A}^\top)^j P B_w B_w^\top P \hat{A}^j K_v^\top \right].  
\end{align}
The infinite sum converges because $\hat{A}$ is a Schur matrix by Lemma~\ref{lem:DARE}. Moreover, the difference between the optimal finite preview and noncausal costs is:
\begin{align}    
\label{eq:Jdiff}
J_{\infty,p}^*- J_{\infty,nc}^*= 
\mbox{trace} \left[ \sum_{j=p+1}^\infty H K_v (\hat{A}^\top)^j P B_w B_w^\top P \hat{A}^j K_v^\top \right]. 
\end{align}
The convergence $J_{\infty,p}^*\to J_{\infty, nc}^*$ follows because $\hat{A}$ is a Schur matrix.
Finally, define $X:=\sum_{j=0}^\infty  (\hat{A}^\top)^j P B_w B_w^\top P \hat{A}^j$.   $X$ is the solution of the discrete-time Lyapunov equation 
\eqref{eq:NCLyap} (See Section 21.1 in \cite{Zhou1996Robust}). This yields the expression for the optimal cost in \eqref{eq:JinfNCSimple}.
\end{proof}

A consequence of the previous results is that the optimal
finite preview cost converges geometrically to the optimal noncausal cost as $p\to \infty$. This is formally stated next.
        
\begin{corollary}
    \label{cor:CostBount}
Let $w$ be an IID sequence with $\E[w_i]=0$ and $\E[w_iw_j^{\top}]=\delta_{ij} \, I$ for $i,j\in \{0,1,\ldots\}$. The difference between the optimal finite
preview and noncausal costs is:
\begin{align}    
\label{eq:JinfDifference}
    J_{\infty,p}^*-J_{\infty,nc}^*= \mbox{trace} \left[ Y (\hat{A}^\top)^{p+1} X  (\hat{A})^{p+1}  \right]. 
\end{align}
where $Y \coloneq K_v^\top H K_v \succeq 0$.
This difference is bounded by:
\begin{align}
    \label{eq:JinfErrorBound}
    \lambda_{\text{min}}(X^\frac{1}{2})^2\lambda_{\text{min}}(Y^\frac{1}{2})^2\|\hat{A}^{p+1}\|_F^2 \leq J_{\infty,p}^*-J_{\infty,nc}^*
\leq \|X^\frac{1}{2}\|_F^2 \|Y^\frac{1}{2}\|_F^2 \|\hat{A}^{p+1}\|_F^2.
\end{align}
\end{corollary}
\begin{proof}
    The difference between the optimal finite preview and noncausal costs is given in Equation~\ref{eq:Jdiff}.
    This expression can be rewritten as follows:
    \begin{align*}    
        J_{\infty,p}^*- J_{\infty,nc}^*= 
        \mbox{trace}\left[ Y  (\hat{A}^\top)^{p+1}
        \left( \sum_{j=0}^\infty (\hat{A}^\top)^j P B_w B_w^\top P \hat{A}^j \right) \hat{A}^{p+1} \right]. 
    \end{align*}
    The summation in parentheses is, by definition, equal to $X$ and this yields Equation~\ref{eq:JinfDifference}.
    Next, rewrite this expression in terms of the Frobenius norm:
    \eqref{eq:JinfDifference}:
    \begin{align*}
         J_{\infty,p}^*-J_{\infty,nc}^*
           =\bigl\|X^{1/2}(\hat{A})^{p+1}Y^{1/2}\bigr\|_F^2.
    \end{align*}
    We obtain the upper bound in \eqref{eq:JinfErrorBound} using the submultiplicative property of the Frobenius norm: $\|PQ\|_F \le \|P\|_F \, \|Q\|_F$ for any two matrices $P$ and $Q$ of appropriate dimensions.
    We obtain the lower bound in \eqref{eq:JinfErrorBound} 
    using the following property:
    $\| PQ\|_F \ge \lambda_{\text{min}}(P)\|Q\|_F$ where $P\succeq 0$ and $Q$ is any matrix of appropriate dimension.
\end{proof}

The spectral radius formula gives the asymptotic
approximation $\lim_{k\to \infty} \|\hat{A}^k \|^{\frac{1}{k}} = \rho(\hat{A})$. Thus the asymptotic bounds in \eqref{eq:JinfErrorBound} as $p\to \infty$ are approximated by:
\begin{align}
c_1  \, \rho\left( \hat{A}\right)^{2(p+1)} 
\lesssim J_{\infty,p}^*-J_{\infty,nc}^* \lesssim    
c_2\,  \rho\left( \hat{A}\right)^{2(p+1)} 
\end{align}
where $c_1:= \lambda_{\text{min}}(X)\lambda_{\text{min}}(Y)$
and $c_2 := \mbox{trace}(X) \mbox{trace}(Y)$ are constants independent of $p$.  This implies  that the finite preview cost converges to the noncausal cost geometrically with $p$.

\section{Example}
\label{sec:example}
This section presents two numerical examples to illustrate the proposed stochastic LQR framework with disturbance preview. 
The first example considers a linear time-invariant (LTI) system and is included in the journal version of the paper to demonstrate the basic structure and properties of the finite-horizon and infinite-horizon solutions. 
The second example considers a linear time-varying (LTV) mass–spring–damper system with time-varying dynamics and cost parameters. This second example illustrates the applicability of the proposed approach to time-varying settings.

\subsection{Boeing 747}
Consider a simplified LTI model for the longitudinal dynamics of a Boeing 747 linearized at one steady, level flight condition (Problem 17.6 in \cite{Boyd_Vandenberghe_2018}):
\begin{equation}\label{eq 16}
x_{t+1}=A \, x_t+B_u \, u_t + w_t
\end{equation}
where the state matrices are:
\begin{align*}    
A=\bsmtx
0.99 & 0.03 & -0.02 & -0.32 \\
0.01 & 0.47 & 4.7 & 0 \\
0.02 & -0.06 & 0.4 & 0 \\
0.01 & -0.04 & 0.72 & 0.99
\esmtx \mbox{ and }
B_u =\bsmtx
0.01 & 0.99 \\
-3.44 & 1.66 \\
-0.83 & 0.44\\
-0.47 & 0.25
\esmtx.
\end{align*}
We consider the IH and FH stochastic LQR problems with constant cost matrices $Q=I$ and $R=I$ at each step. This example was used in prior work on competitive ratio
\cite{goel22TAC} and additive regret \cite{sabag21ACC,sabag21arXiv}.


Figure~\ref{fig:FHCost} shows the performance achieved
by the optimal FH controller on the horizon $T=100$ for previews $p=5$ and 20 (solid lines). A closed-loop simulation was performed with a random disturbance. The plot shows the per-step cost averaged from time 0 to time $t$ versus the time $t$. This is compared against the performance of the optimal $H_2$ controller computed on the IH (dashed lines). This $H_2$ controller is computed using a system augmented with delays to account for the disturbance preview as in \cite{Ting2021WindTunnelStudy, MoeljaMeinsma2006H2Control, Yim2011DesignOfPreviewController, ROH1999StochasticOptimalPreviewControl, Louam1988OptimalControl, Tomizuka1975OptimalDiscreteFinitePreviewProblems}. This full information $H_2$ controller is computed using \texttt{h2syn} in Matlab. The optimal FH controller and IH $H_2$ controller have similar performance for $t \le 60$. In fact, the FH and IH gains are similar at the beginning of the simulation. However, the FH gains deviate significantly from the IH gains near the end of the horizon.  The FH gains 
are optimal for the horizon and hence the FH controller yields improved performance (lower average cost) by the end of the horizon.


\begin{figure}[ht]
 \centering
    \includegraphics[width=0.47\textwidth]{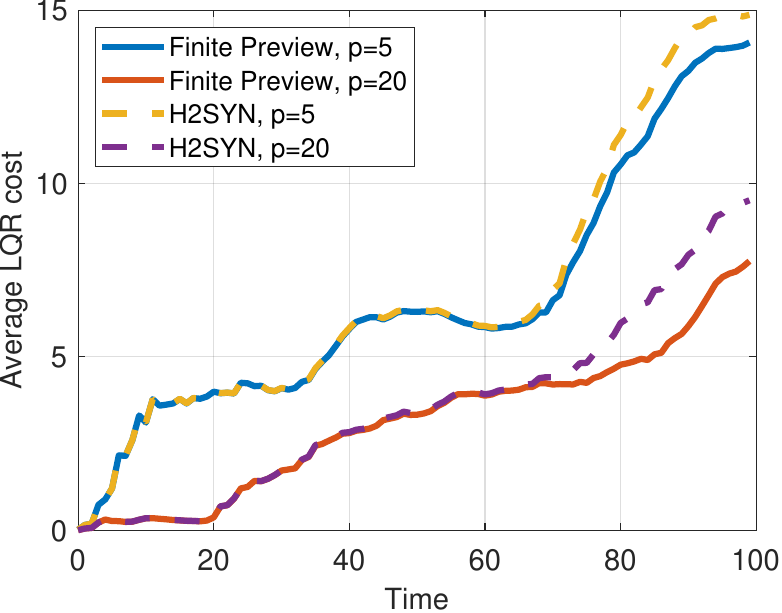}
    \caption{Comparison of average cost obtained by optimal FH controller and optimal $H_2$ controller computed via an augmented system.}
     \label{fig:FHCost}
\end{figure}

\begin{figure}[ht]
 \centering
    \includegraphics[width=0.47\textwidth]{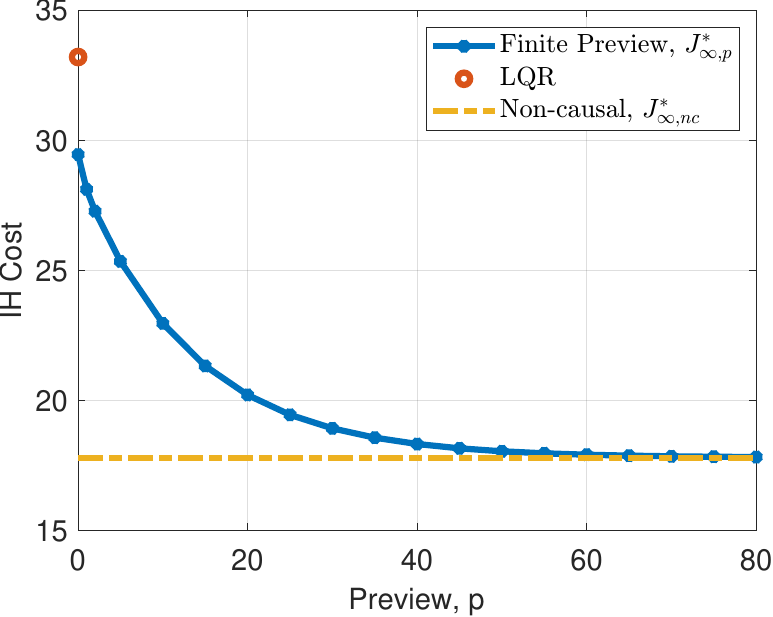}
    \caption{Comparison of IH cost versus preview horizon $p$.}
    \label{fig:IHCost}
\end{figure}

Figure~\ref{fig:IHCost} shows the optimal IH cost versus the preview horizon $p$ (blue solid).
The cost achieved by the finite preview controller is computed using the expression in Equation~\ref{eq:JinfFinal}.  The plot also shows the cost achieved by the Linear Quadratic Regulator (LQR) state feedback: $u_t = -K_x x_t$ where $K_x=H^{-1} B_u^\top P A$ is the same gain as used in the finite preview controller. The LQR state feedback achieves a cost equal to trace$[PB_wB_w^\top] \approx 33.2$ (red circle).  The full information controller $u_t = -K_x x_t - K_w w_t$ corresponds to the case of $p=0$.  This improves on the LQR performance and reduces the cost to $\approx 29.4$. Adding disturbance preview $(p>1)$ further decreases the cost. In fact, the cost monotonically decreases with increasing preview. 

Figure~\ref{fig:IHCost} also shows the performance of the optimal noncausal controller (Theorem~\ref{thm:IHNoncausal}) for comparison.  The cost for the optimal noncausal controller depends on the specific disturbance input. The yellow dashed line in Figure~\ref{fig:IHCost} corresponds to the expected (average) cost of the noncausal controller assuming white noise disturbances. This was computed using the expression given in Theorem~\ref{thm:CostConv}.
The noncausal controller has full knowledge of the disturbance and yields the lowest possible cost of $\approx 17.8$.  The performance of the finite preview controller converges, as $p\to \infty$, to the performance of the optimal noncausal controller. This is expected based on  
the discussion in Section~\ref{sec:NCresult}.

Finally, Figure~\ref{fig:IHRelativeCost} shows the optimal IH relative error of the cost for the finite preview controller compared to the cost for the optimal noncausal controller. This is normalized by the noncausal cost and plotted on a logarithmic scale against the preview horizon $p$ (blue solid). The plot uses the same cost results as in Figure~\ref{fig:IHCost}. The plot is linear on a log scale in agreement with the geometric convergence given in Corollary ~\ref{cor:CostBount}. The slope on the log scale should be approximately $2\log\rho(A)$. as $p\to \infty$.

\begin{figure}[ht]
 \centering
    \includegraphics[width=0.47\textwidth]{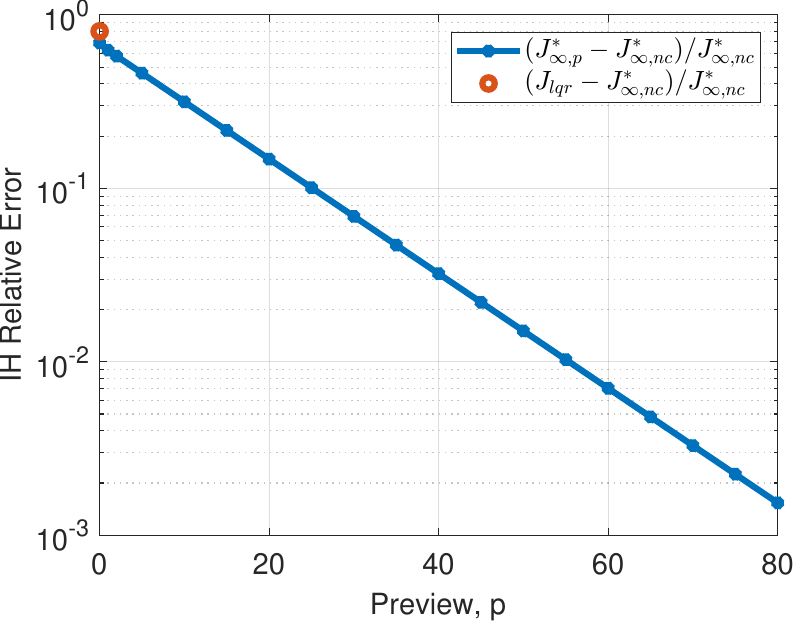}
    \caption{Comparison of IH relative error on a log scale versus preview horizon $p$.}
    \label{fig:IHRelativeCost}
\end{figure}

\subsection{Mass-spring-damper}

Consider a time-varying mass--spring--damper system with fixed damping ratio $\zeta$ and time-varying natural frequency $\omega_n(k)$. 
The system is modeled directly in discrete time as a second-order linear time-varying system
\begin{equation}\label{eq:LTVmsd_dt}
x_{k+1} = A_k x_k + B_k u_k + E_k w_k, \qquad k = 0,1,\ldots,T-1,
\end{equation}
where $x_k = [\,q_k\;\; \dot{q}_k\,]^\top$ denotes the displacement and velocity.
The time-varying system matrices are given by
\[
A_k =
\begin{bmatrix}
1 & T_s \\
- T_s \omega_n(k)^2 & 1 - 2 T_s \zeta \omega_n(k)
\end{bmatrix},
\qquad
B_k = E_k =
\begin{bmatrix}
0 \\
T_s
\end{bmatrix}.
\]
This model corresponds to an Euler discretization of an equivalent continuous-time, mass-spring-damper system with sample time $T_s$. The simulation uses $T_s=0.02$s and
$\zeta=0.7$.  Moreover, the natural frequency is given by:
\begin{align}
    w_n(k) = 1.5 + 0.5  \sin\left(
    \frac{2\pi T_s k}{10} \right);  
\end{align}
Finally, the disturbance $w_k$ is modeled as i.i.d.\ Gaussian noise.

We consider a finite-horizon stochastic LQR problem with time-varying state costs.
Specifically, the input cost is fixed as $R_k = 0.1$, while the state cost $Q_k$ is chosen as $\mathrm{diag}(10,1)$ in the first half of the horizon and $\mathrm{diag}(50,1)$ in the second half. This penalizes the displacement more heavily later in the horizon.
We compute the optimal finite-horizon preview controller for several preview horizons $p$ and evaluate its performance using closed-loop simulations.

Figure~\ref{fig:LTVFHCost} shows the average per-step cost from step $0$ to step $k$ versus the step index $k$ for preview horizons $p\in\{1,5,20\}$. The performance of the optimal noncausal controller is also shown.
As expected, increasing the preview horizon leads to improved performance (lower average cost). In particular, the performance with $p=20$ is already very close to that of the noncausal controller. The noncausal controller, which has access to full disturbance information, achieves the lowest cost and serves as a benchmark.

\begin{figure}[ht]
 \centering
    \includegraphics[width=0.47\textwidth]{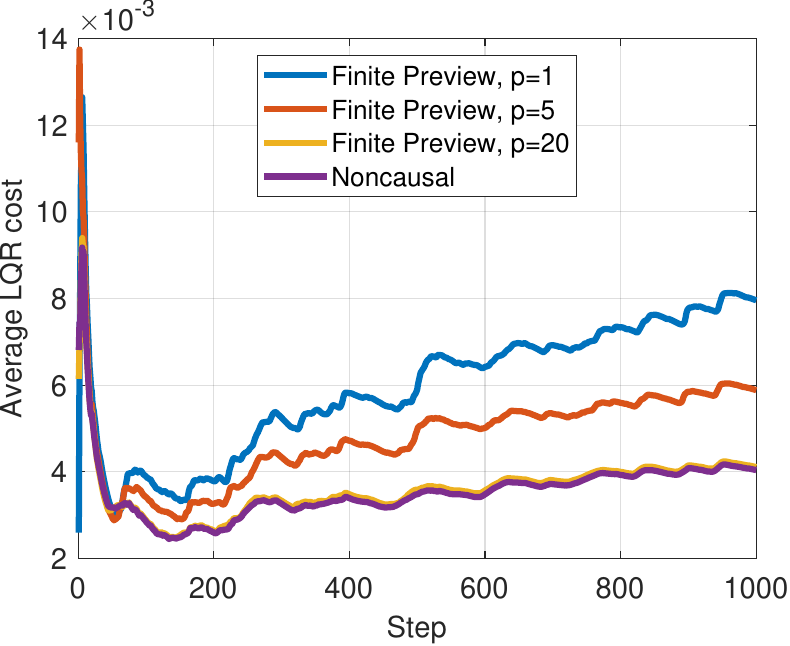}
    \caption{Comparison of average cost obtained by optimal FH controller and optimal noncausal controller in LTV system.}
    \label{fig:LTVFHCost}
\end{figure}

\section{Conclusion}
This paper presented solutions to the discrete-time, stochastic LQR problem with $p$ steps of disturbance preview information where $p$ is finite. Solutions were presented for both the finite horizon and infinite horizon problems.  Moreover, we provided a proof for the principle of optimality that relies only on the assumption of  nested information structure.  This is of independent interest.  Finally, we compared our solutions to several existing results for optimal control with preview information. Future work will explore using our finite-preview controller as a baseline for regret-based design.
Additionally, we will consider extensions to cases where the dynamic models and cost matrices are time-varying and only known with finite preview, in contrast to the full future knowledge assumed in this paper.

\section{Acknowledgments}

This material is based upon work supported by the National Science Foundation under Grant No. 2347026. Any opinions, findings, and conclusions or recommendations expressed in this material are those of the author(s) and do not necessarily reflect the views of the National Science Foundation.

\bibliographystyle{IEEEtran}
\bibliography{reference}

@book{horn2013matrix,
  title={Matrix analysis},
  author={R.A. Horn  and C.R. Johnson},
  year={2013},
  publisher={Cambridge university press}
}

@ARTICLE{ozdemir13,
  author={A.A. Ozdemir and P. Seiler and G.J. Balas},
  journal={IEEE Transactions on Control Systems Technology}, 
  title={Design Tradeoffs of Wind Turbine Preview Control}, 
  year={2013},
  volume={21},
  number={4},
  pages={1143-1154},
  doi={10.1109/TCST.2013.2261069}}

@article{schlipf13,
  title={Nonlinear model predictive control of wind turbines using LIDAR},
  author={D. Schlipf and D.J. Schlipf and M. K{\"u}hn},
  journal={Wind energy},
  volume={16},
  number={7},
  pages={1107--1129},
  year={2013},
  publisher={Wiley Online Library}
}

@inproceedings{scholbrock16,
  title={Lidar-enhanced wind turbine control: Past, present, and future},
  author={A. Scholbrock and P. Fleming and D. Schlipf and A. Wright and K. Johnson and N. Wang},
  booktitle={American Control Conference},
  pages={1399--1406},
  year={2016},
  organization={IEEE}
}

@book{Zhou1996Robust,
  title={Robust and optimal control},
  author={K. Zhou and J. C. Doyle and K. Glover},
  year={1996},
  publisher={Prentice hall}
}

@book{bernstein2018scalar,
  title={Scalar, vector, and matrix mathematics: theory, facts, and formulas-revised and expanded edition},
  author={D. Bernstein},
  year={2018},
  publisher={Princeton Univ. Press}
}

@book{anderson2007optimal,
  title={Optimal control: linear quadratic methods},
  author={B.D.O. Anderson and J.B. Moore},
  year={1990},
  publisher={Prentice-Hall}
}

@book{chung2001,
  title={A Course in Probability Theory},
  author={K.L. Chung},
  year={2001},
  publisher={Academic Press}
}

@book{Boyd_Vandenberghe_2018, 
place={Cambridge}, 
title={Introduction to Applied Linear Algebra: Vectors, Matrices, and Least Squares}, publisher={Cambridge Univ. Press}, 
author={S. Boyd and L. Vandenberghe}, 
year={2018}}

@article{MoeljaMeinsma2006H2Control,
author = {A.A. Moelja and G. Meinsma},
title = {{$H_2$} control of preview systems},
journal = {Automatica},
volume = {42},
number = {6},
pages = {945-952},
year = {2006},
issn = {0005-1098}
}

@article{MarroZattoni2005H2OptimalRejection,
title = {{$H_2$}-optimal rejection with preview in the continuous-time domain},
author={G. Marro and E. Zattoni},
journal = {Automatica},
volume = {41},
number = {5},
pages = {815-821},
year = {2005},
publisher={Elsevier}
}

@article{Hazell2010FrameworkForDiscreteTimeH2PreviewControl,
    author = {A.J. Hazell and D.J.N. Limebeer},
    title = {A Framework for Discrete-Time {$H_2$} Preview Control},
    journal = {Journal of Dynamic Systems, Measurement, and Control},
    volume = {132},
    number = {3},
    pages = {031005},
    year = {2010},
    month = {04}
}

@INPROCEEDINGS{SentouhEtAl2011TheH2OptimalPreviewController,
  author={C. Sentouh and B. Soualmi and J.C. Popieul and S. Debernard},
  booktitle={International Conference on Intelligent Transportation Systems}, 
  title={The {$H_2$}-optimal preview controller for a shared lateral control}, 
  year={2011},
  pages={1452-1458},
}

@article{Ting2021WindTunnelStudy,
  title={Wind Tunnel Study of Preview {$H_2$} and {$H_{\infty}$} Control for Gust Load Alleviation for Flexible Aircraft},
  author={K.Y. Ting and M. Mesbahi and E. Livne and K. A. Morgansen},
  journal={AIAA SciTech},
  year={2022}}

@INPROCEEDINGS{Zattoni2026H2OptimalDecouplingWithPreview,
  author={E. Zattoni},
  booktitle={American Control Conference}, 
  title={{$H_2$}-optimal decoupling with preview: a dynamic feedforward solution based on factorization techniques}, 
  year={2006},
  volume={},
  number={},
  pages={316-320},
 }

@ARTICLE{Kojima1999LQPreviewSynthesis,
  author={A. Kojima and S. Ishijima},
  journal={IEEE Transactions on Automatic Control}, 
  title={{LQ} preview synthesis: optimal control and worst case analysis}, 
  year={1999},
  volume={44},
  number={2},
  pages={352-357},
 }

@ARTICLE{Tomizuka1975OptimalContinuousFinitePreviewProblem,
  author={Tomizuka, M.},
  journal={IEEE Trans. on Automatic Control}, 
  title={Optimal continuous finite preview problem}, 
  year={1975},
  volume={20},
  number={3},
  pages={362-365},
}

@article{LINDQUIST1968OnOptimalStochasticControlWithSmoothedInformation,
title = {On optimal stochastic control with smoothed information},
journal = {Information Sciences},
volume = {1},
number = {1},
pages = {55-85},
year = {1968},
author = {A. Lindquist},
}

@article{Tomizuka1975OptimalDiscreteFinitePreviewProblems,
    author = {Tomizuka, M. and Whitney, D. E.},
    title = "{Optimal Discrete Finite Preview Problems (Why and How Is Future Information Important?)}",
    journal = {Journal of Dynamic Systems, Meas., and Control},
    volume = {97},
    number = {4},
    pages = {319-325},
    year = {1975},
    month = {12},
}

@article{Tomizuka1976OptimalLinearPreviewControl,
    author = {Tomizuka, M.},
    title = "{Optimum Linear Preview Control With Application to Vehicle Suspension—Revisited}",
    journal = {Journal of Dyn. Systems, Meas., and Control},
    volume = {98},
    number = {3},
    pages = {309-315},
    year = {1976},
    month = {09},
}

@article{Hac1992OptimalLinearPreviewControl,
author = {A. Ha\'{c}},
title = {Optimal Linear Preview Control of Active Vehicle Suspension},
journal = {Vehicle System Dynamics},
volume = {21},
number = {1},
pages = {167--195},
year = {1992},
publisher = {Taylor \& Francis},
}

@article{ROH1999StochasticOptimalPreviewControl,
title = {STOCHASTIC OPTIMAL PREVIEW CONTROL OF AN ACTIVE VEHICLE SUSPENSION},
author = {H.-S. Roh and Y. Park},
journal = {J. of Sound and Vibration},
volume = {220},
number = {2},
pages = {313-330},
year = {1999},
}

@article{Louam1988OptimalControl,
author = {N. Louam and D. A. Wilson and R. S. Sharp},
title = {Optimal Control of a Vehicle Suspension Incorporating the Time Delay between Front and Rear Wheel Inputs},
journal = {Vehicle System Dynamics},
volume = {17},
number = {6},
pages = {317--336},
year = {1988},
publisher = {Taylor \& Francis},
}

@techreport{Peng1991optimal,
  title={Optimal preview control for vehicle lateral guidance},
  author={H. Peng and M. Tomizuka},
  institution = {California {PATH}},
  number = {{UCB-ITS-PRR-91-16}},
  year={1991}
}

@ARTICLE{Yim2011DesignOfPreviewController,
  author={S. Yim},
  journal={IEEE Transactions on Vehicular Technology}, 
  title={Design of a Preview Controller for Vehicle Rollover Prevention}, 
  year={2011},
  volume={60},
  number={9},
  pages={4217-4226},
}

@article{YOSHIMURA1993AnActiveSuspensionModel,
title = {An Active Suspension Model For Rail/Vehicle Systems With Preview and Stochastic Optimal Control},
author = {T. Yoshimura and K. Edokoro and N. Ananthanarayana},
journal = {J. of Sound and Vibration},
volume = {166},
number = {3},
pages = {507-519},
year = {1993},
}

@article{MARZBANRAD2004StochasticOptimalPreviewControl,
title = {Stochastic optimal preview control of a vehicle suspension},
author = {J. Marzbanrad and G. Ahmadi and H. Zohoor and Y. Hojjat},
journal = {Journal of Sound and Vibration},
volume = {275},
number = {3},
pages = {973-990},
year = {2004},
}

@article{TomizukaRosenthal1979OnTheOptiamlDigitalStateVectorFeedbackController,
    author = {Tomizuka, M. and Rosenthal, D.E.},
    title = "{On the Optimal Digital State Vector Feedback Controller With Integral and Preview Actions}",
    journal = {Journal of Dynamic Systems, Meas., and Control},
    volume = {101},
    number = {2},
    pages = {172-178},
    year = {1979},
    month = {06},
}

@article{TomizukaFung1980DesignOfDigitalFeedforwardPreviewControllers,
    author = {Tomizuka, M. and Fung, D.H.},
    title = "{Design of Digital Feedforward/Preview Controllers for Processes With Predetermined Feedback Controllers}",
    journal = {Journal of Dynamic Systems, Meas., and Control},
    volume = {102},
    number = {4},
    pages = {218-225},
    year = {1980},
    month = {12},
}

@article{HashikuraEtAl2020OnImplementationOfH2PreviewOutputFeedbackLaw,
author = {K. Hashikura and R. Hotchi and A. Kojima and T. Masuta},
title = {On implementations of {$H_2$} preview output feedback law with application to {LFC} with load demand prediction},
journal = {International Journal of Control},
volume = {93},
number = {4},
pages = {844--857},
year = {2020},
publisher = {Taylor \& Francis},
}

@book{hassibi99,
  author={B. Hassibi and A.H. Sayed and T. Kailath},
  year={1999},
  publisher={SIAM},
  title={Indefinite-Quadratic estimation and control: a unified approach to {$H_2$} and {$H_{\infty}$} theories}
}

@inproceedings{sabag21ACC,
  title={Regret-optimal controller for the full-information problem},
  author={Sabag, O. and Goel, G. and Lale, S. and Hassibi, B.},
  booktitle={American Control Conference},
  pages={4777--4782},
  year={2021}
}

@article{sabag22arXiv,
  title={Optimal Competitive-Ratio Control},
  author={O. Sabag and S. Lale and B. Hassibi},
  journal={arXiv:2206.01782},
  year={2022}
}

@article{sabag21arXiv,
  title={Regret-optimal full-information control},
  author={Sabag, O. and Goel, G. and Lale, S. and Hassibi, B.},
  journal={arXiv:2105.01244},
  year={2021}
}

@article{goel22TAC,
  title={Competitive control},
  author={Goel, G. and Hassibi, B.},
  journal={IEEE Transactions on Automatic Control},
  year={2022},
  volume={68},
  number={9},
  pages={5162-5173},
  publisher={IEEE}
}

@ARTICLE{goel20arXiv,
  author={Goel, G. and Hassibi, B.},
  journal={IEEE Transactions on Automatic Control}, 
  title={Regret-Optimal Estimation and Control}, 
  year={2023},
  volume={68},
  number={5},
  pages={3041-3053},
  doi={10.1109/TAC.2023.3253304}}

@inproceedings{goel22CDC,
  title={The power of linear controllers in {LQR} control},
  author={Goel, G. and Hassibi, B.},
  booktitle={IEEE Conference on Decision and Control},
  pages={6652--6657},
  year=2022,
  organization={IEEE}
}

@article{liu2024robust,
  title={Robust regret optimal control},
  author={Liu, J. and Seiler, P.},
  journal={Int. Journal of Robust and Nonlinear Control},
  volume={34},
  number={7},
  pages={4532--4553},
  year={2024},
  publisher={Wiley Online Library}
}

@inproceedings{Zhang2021,
  title={On the regret analysis of online LQR control with predictions},
  author={Zhang, Runyu and Li, Yingying and Li, Na},
  booktitle={2021 American Control Conference (ACC)},
  pages={697--703},
  year={2021},
  organization={IEEE}
}

@misc{Lin2021,
      title={Perturbation-based Regret Analysis of Predictive Control in Linear Time Varying Systems}, 
      author={Y. Lin and Y. Hu and H. Sun and G. Shi and G. Qu and A. Wierman},
      year={2021},
      eprint={2106.10497},
      archivePrefix={arXiv},
      primaryClass={math.OC},
      url={https://arxiv.org/abs/2106.10497}, 
}

\end{document}